\documentclass[12pt,a4paper]{article}
\usepackage{latexsym,amsfonts,amsmath,graphics,amssymb}
\usepackage{verbatim}
\usepackage{graphicx}

\usepackage[colorlinks=true,linkcolor=black,citecolor=black]{hyperref}
\usepackage[active]{srcltx}

\newtheorem{theorem}{Theorem}
\newtheorem{lemma}{Lemma}
\newtheorem{corollary}{Corollary}
\newtheorem{proposition}{Proposition}

\newtheorem{algorithm}{Algorithm}

\newenvironment{proof}{\begin{trivlist}
    \item[\hskip\labelsep{\it Proof.}]}{$\hfill\Box$\end{trivlist}}

\newcommand{\To}{\rightarrow}

\newcommand{\bsgamma}{\boldsymbol{\gamma}}

\newcommand{\bsk}{\boldsymbol{k}}

\newcommand{\bsx}{\boldsymbol{x}}
\newcommand{\bsh}{\boldsymbol{h}}
\newcommand{\bsg}{\boldsymbol{g}}

\newcommand{\bsy}{\boldsymbol{y}}

\newcommand{\uu}{\mathfrak{u}}

\newcommand{\icomp}{\mathtt{i}}
\newcommand{\bszero}{\boldsymbol{0}}
\newcommand{\rd}{\,\mathrm{d}}
\newcommand{\NN}{\mathbb{N}}
\newcommand{\ZZ}{\mathbb{Z}}

\newcommand{\RR}{\mathbb{R}}

\renewcommand{\pmod}[1]{\,(\bmod\,#1)}

\newcommand{\abs}[1]{\left\vert#1\right\vert}

\allowdisplaybreaks

\usepackage[usenames]{color}
\definecolor{darkblue}{RGB}{40,0,200}

\begin{document}

\title{\scshape  Numerical integration in  $\log$-Korobov and $\log$-cosine spaces}

\author{Josef Dick\thanks{
Josef Dick is the recipient of an Australian Research Council Queen Elizabeth II Fellowship (project number DP1097023).},
Peter Kritzer\thanks{P. Kritzer is supported by the Austrian Science Fund (FWF):
Project F5506-N26, which is a part of the Special Research Program "Quasi-Monte Carlo Methods: Theory and Applications".},
Gunther Leobacher\thanks{G. Leobacher is supported by the Austrian Science Fund (FWF): Project F5508-N26,
which is a part of the Special Research Program "Quasi-Monte Carlo Methods: Theory and Applications".},
Friedrich Pillichshammer\thanks{F. Pillichshammer is supported by the Austrian Science Fund (FWF): Project F5509-N26,
which is a part of the Special Research Program "Quasi-Monte Carlo Methods: Theory and Applications".}}

\date{}
\maketitle

\begin{abstract}
Quasi-Monte Carlo (QMC) rules are equal weight quadrature rules for approximating integrals over the $s$-dimensional 
unit cube $[0,1]^s$. One line of research studies the integration error of functions in the unit ball of so-called Korobov spaces, 
which are Hilbert spaces of periodic functions on $[0,1]^s$ with square integrable partial mixed derivatives of order $\alpha$. 
Using Parseval's identity, this smoothness can be defined for all real numbers $\alpha > 1/2$. In this setting, the condition 
$\alpha> 1/2$ is necessary as otherwise the Korobov space contains discontinuous functions for which function evaluation is not well defined.

This paper is concerned with more precise endpoint estimates of the integration error using QMC rules for Korobov spaces with 
$\alpha$ arbitrarily close to $1/2$. To obtain such estimates we introduce a $\log$-scale for functions with smoothness close to 
$1/2$, which we call $\log$-Korobov spaces. We show that lattice rules can be used to obtain an integration error of order 
$\mathcal{O}(N^{-1/2} (\log N)^{-\mu(1-\lambda)/2})$ for any $1/\mu <\lambda \le 1$, where $\mu>1$ is a certain power in the $\log$-scale. 

A second result is concerned with tractability of numerical integration for weighted Korobov spaces with product weights 
$(\gamma_j)_{j \in \mathbb{N}}$. Previous results have shown that if $\sum_{j=1}^\infty \gamma_j^\tau < \infty$ for some 
$1/(2\alpha) < \tau \le 1$ one can obtain error bounds which are independent of the dimension. In this paper we give a more refined 
estimate for the case where $\tau$ is close to $1/(2 \alpha)$, namely we show dimension independent error bounds under the condition that 
$\sum_{j=1}^\infty \gamma_j \max\{1, \log \gamma_j^{-1}\}^{\mu(1-\lambda)} <
\infty$ for some $1/\mu < \lambda \le 1$. The essential tool in our analysis
is a $\log$-scale Jensen's inequality.

The results described above also apply to integration in $\log$-cosine spaces using tent-transformed lattice rules. 
\end{abstract}




\section{Classical Korobov spaces}

We study quasi-Monte Carlo (QMC) integration rules
\begin{equation*}
\frac{1}{N} \sum_{n=0}^{N-1} f(\boldsymbol{x}_n) \approx \int_{[0,1]^s} f(\boldsymbol{x}) \,\mathrm{d} \boldsymbol{x}
\end{equation*}
for the approximation of $s$-dimensional integrals where the quadrature point set 
$P = \{ \boldsymbol{x}_0, \boldsymbol{x}_1, \ldots, \boldsymbol{x}_{N-1} \} \subseteq [0,1]^s$ 
is deterministically constructed. We use the notation $\boldsymbol{x}_n = (x_{n,1}, x_{n,2}, \ldots, x_{n,s})$. 
In particular we focus on so-called lattice rules, which are QMC rules using the point set $P = \{\bsx_0, \bsx_1,\ldots, \bsx_{N-1}\}$ given by
\begin{equation}\label{lattice_points}
\bsx_n = \left(\left\{n g_1/N \right\}, \ldots, \left\{n g_s/N \right\} \right), \quad n = 0, 1, \ldots, N-1,
\end{equation}
where $g_1, \ldots, g_s \in \{1, 2, \ldots, N-1\}$ and $N$ is a prime number. 
The notation $\{\cdot \}$ in~\eqref{lattice_points} denotes the fractional part
$\{x \} = x - \lfloor x \rfloor$ for nonnegative numbers $x \in \mathbb{R}$. One of the main
questions in the study of lattice rules is the construction of a good
generating vector $\bsg = (g_1, g_2, \ldots, g_s)$; see \cite{CKN06, D04,
DKLP15, K03, KSS12, NC06, NC06b, SKJ02}.

For this, let $H$ be a Hilbert space of functions $f: [0,1]^s \to \mathbb{R}$ with inner product 
$\langle \cdot, \cdot \rangle_H$ and corresponding norm $\|\cdot \|_H$. The worst-case error
\begin{equation*}
e(H, P) = \sup_{f \in H, \|f\|_H \le 1} \left|\int_{[0,1]^s} f(\boldsymbol{x}) \,\mathrm{d} \boldsymbol{x} - 
\frac{1}{N} \sum_{n=0}^{N-1} f(\boldsymbol{x}_n) \right|
\end{equation*}
serves as a criterion for the quality of the quadrature point set $P$. Frequently one studies reproducing kernel Hilbert spaces 
$H_K$ which have a reproducing kernel $K:[0,1]^s \times [0,1]^s \to \mathbb{R}$. 
See \cite{Aron} for more information on reproducing kernel Hilbert spaces and \cite{DKS13, DP10} for background on reproducing kernels 
in the context of numerical integration. The reproducing kernel has the properties that $K(\cdot, \boldsymbol{y}) \in H_K$ for all 
$\boldsymbol{y} \in [0,1]^s$ and
\begin{equation*}
f(\boldsymbol{y}) = \langle f, K(\cdot, \boldsymbol{y}) \rangle_{H_K} \quad \mbox{for all $f\in H_K$ and all $\bsy\in[0,1]^s$,}
\end{equation*}
where $\langle \cdot, \cdot \rangle_{H_K}$ is the inner product in $H_K$. 
One class of reproducing kernel Hilbert spaces that has been studied intensively, are the Korobov spaces \cite{DKS13}.  Those consist of functions on 
$[0,1]^s$ which have square integrable derivatives up to order $\alpha$ in each
variable. Here, the reproducing kernel is given by
\begin{equation}\label{Korobov_kernel}
K_{\alpha,\bsgamma}(\boldsymbol{x}, \boldsymbol{y}) = 
\sum_{\boldsymbol{k} \in \mathbb{Z}^s} w_{\alpha,\boldsymbol{\gamma}}(\boldsymbol{k}) \mathrm{e}^{2\pi \mathrm{i} 
\boldsymbol{k} \cdot (\boldsymbol{x} - \boldsymbol{y})},
\end{equation}
where $w_{\alpha, \boldsymbol{\gamma}}(\boldsymbol{k}) = \prod_{j=1}^s w_{\alpha, \gamma_j}(k_j)$ and
\[
w_{\alpha, \gamma_j}(k_j) = 
\begin{cases} 
1 & \mbox{if $k_j=0$,} \\
\gamma_j |k_j|^{-2\alpha} & \mbox{otherwise.}\\ 
\end{cases}
\]
The inner product in this space is given by
\begin{equation*}
\langle f, g \rangle_{H_{K_{\alpha,\bsgamma}}} = \sum_{\boldsymbol{k} \in \mathbb{Z}^s} \widehat{f}(\boldsymbol{k}) 
\overline{\widehat{g}(\boldsymbol{k})} w^{-1}_{\alpha,\bsgamma}(\boldsymbol{k}),
\end{equation*}
where $\overline{a}$ denotes the complex conjugate of a complex number $a$ and where $\widehat{f}(\bsk)$ and 
$\widehat{g}(\bsk)$ are the $\bsk$-th Fourier coefficients of $f$ and $g$ respectively,
\begin{equation*}
\widehat{f}(\bsk) = \int_{[0,1]^s} f(\bsx) \mathrm{e}^{-2\pi \mathrm{i} \bsk \cdot \bsx} \,\mathrm{d} \bsx.
\end{equation*}
The sequence $\boldsymbol{\gamma} = (\gamma_j)_{j \in \mathbb{N}}$ of nonnegative real numbers are called weights,
and they determine the dependence of the variables on the dimension \cite{NW10}.

The worst-case error for numerical integration in a reproducing kernel Hilbert space is given by the formula (cf.~\cite{Hick} or \cite[Proposition~2.11]{DP10})
\begin{align*}
e(H_K, P) = & \int_{[0,1]^s} \int_{[0,1]^s} K(\boldsymbol{x}, \boldsymbol{y}) \,\mathrm{d} \boldsymbol{x} \,\mathrm{d} \boldsymbol{y} - 
\frac{2}{N} \sum_{n=0}^{N-1} \int_{[0,1]^s} K(\boldsymbol{x}, \boldsymbol{x_n}) \,\mathrm{d} \boldsymbol{x} 
+ \frac{1}{N^2} \sum_{n,m=0}^{N-1} K(\boldsymbol{x}_n, \boldsymbol{x}_m)\,,
\end{align*}
which, for the Korobov space simplifies to
\begin{equation*}
e(H_{K_{\alpha,\bsgamma}}, P) = -1 + \frac{1}{N^2} \sum_{n=0}^{N-1} \sum_{m=0}^{N-1} K_{\alpha,\bsgamma}(\boldsymbol{x}_n, \boldsymbol{x}_m).
\end{equation*}
If the QMC-rule is a lattice rule, then this formula reduces to
\begin{equation*}
e(H_{K_{\alpha,\bsgamma}}, P)=  \sum_{\bsk \in L^\perp \setminus\{\bszero\}} w_{\alpha,\bsgamma}(\bsk)
\end{equation*}
where the dual lattice $L^\perp$ is given by
\begin{equation}\label{duallat}
L^\perp = \{ \boldsymbol{\ell} \in \mathbb{Z}^s: \boldsymbol{\ell} \cdot \bsg \equiv 0 \pmod{N}\}.
\end{equation}

 In order for the sum in \eqref{Korobov_kernel} to converge absolutely (and therefore for the reproducing kernel and worst-case error to be well defined), 
we need to assume that $\alpha> 1/2$. One of the aims of this paper is to study the endpoint $\alpha$ close to $1/2$ in more detail. 
QMC rules are often applied to functions which do not satisfy the smoothness assumptions usually considered and hence it is 
of interest to study QMC for non-smooth functions. We do so by introducing a $\log$ scale in the definition of the Korobov space and we do the same also for cosine spaces considered in \cite{DNP}. This will be done in Section \ref{sec_log_kor}. In Section \ref{sec_cbc} we
present a component-by-component construction of a lattice rule which  
will be used in Section \ref{sec_err_bnds} to give  bounds on the 
integration error in $\log$-Korobov spaces.

Section \ref{sec_tract} considers tractability of integration in $\log$-Korobov
spaces, Section \ref{sec_ext} concludes with a possible extension of our 
results to iterated $\log$-Korobov  spaces. 

\section{$\log$-Korobov spaces}\label{sec_log_kor}

\subsection{The $\log$-Korobov space}\label{subsec_log_kor}

Throughout this paper $\log$ denotes the natural logarithm.

The $\log$-Korobov space is again a reproducing kernel Hilbert space with kernel of the form \eqref{Korobov_kernel} with the weight 
$w_{\alpha,\bsgamma}$ replaced by $r_{\mu,\bsgamma}(\boldsymbol{k}) = \prod_{j=1}^s r_{\mu,\gamma_j}(k_j)$, where
$$
r_{\mu,\gamma_j}(k_j) =
\begin{cases}
1 & \mbox{if $k_j = 0$}, \\
\gamma_j |k_j|^{-1} (\log (\kappa |k_j|))^{-\mu} & \mbox{otherwise,}
\end{cases}
$$
where $\mu>1$ is a real number, $\bsgamma = (\gamma_j)_{j \in \NN}$ is a sequence of positive real numbers and $\kappa$ is a fixed real number, 
which we assume to satisfy $\kappa \ge \max\{\exp(\mathrm{e}^2),\exp(\mathrm{e}^2 \gamma_j^{1/\mu})\}$ for all $j$ (note that this implies 
$r_{\mu,\gamma_j} (k_j)\le {\rm e}^{-2\mu}<{\rm e}^{-2}$ for any $k_j\neq 0$). 
The reproducing kernel for the $\log$-Korobov space is given by 
$$K_{\log, \mu, \bsgamma}(\boldsymbol{x}, \boldsymbol{y}) = 
\sum_{\boldsymbol{k} \in \mathbb{Z}^s} r_{\mu,\bsgamma}(\boldsymbol{k}) \mathrm{e}^{2\pi \mathrm{i} 
\boldsymbol{k} \cdot (\boldsymbol{x} - \boldsymbol{y})},$$
and we denote the corresponding reproducing kernel Hilbert space by $H_{K_{\log, \mu, \bsgamma}}$. The inner product in this space is given by
\begin{equation*}
\langle f, g \rangle_{H_{K_{\log,\mu,\bsgamma}}} = \sum_{\boldsymbol{k} \in \mathbb{Z}^s} \widehat{f}(\boldsymbol{k}) 
\overline{\widehat{g}(\boldsymbol{k})} r^{-1}_{\mu,\bsgamma}(\boldsymbol{k}),
\end{equation*}

The squared worst-case integration error for functions in the $\log$-Korobov space using lattice rules is given by 
(see for instance \cite[Eq. (15)]{SW01})
\begin{equation}\label{wce}
e^2(H_{K_{\log, \mu, \bsgamma}}, P) = \sum_{\bsk \in L^\perp \setminus\{\bszero\}} r_{\mu,\bsgamma}(\bsk),
\end{equation}
where the dual lattice $L^\perp$ is given by \eqref{duallat}.

In one dimension, functions in a Korobov space $H_{K_{\alpha,\bsgamma}}$ of smoothness $\alpha > 1/2$ satisfy a H\"older condition and therefore have 
smoothness beyond continuity (in higher dimension functions in the Korobov space have finite fractional bounded variation), whereas the functions in the $\log$-Korobov space need not satisfy a H\"older condition. We illustrate this by the following proposition.

\begin{proposition}\label{prop_holder}
The (univariate) function $f$ defined by
$$f(x):=\sum_{k=1}^\infty \frac{1}{k (\log (\kappa k))^\mu} \cos(2 \pi k x) \in H_{K_{\log, \mu, 1}}$$
is uniformly continuous but not H\"older continuous.
\end{proposition}

\begin{proof}
Let $\beta>0$ and consider the sequence
$m^{\beta}(f(0)-f(m^{-1}) )$ for $m=1,2,\ldots$. We have
\begin{align*}
m^{\beta}\left(f(0)-f\left(m^{-1} \right)\right)
&= m^{\beta}\sum_{k=1}^\infty \frac{1}{k (\log(\kappa k))^\mu}
\left(1-\cos\left(\frac{2\pi k}{m}\right)\right) \\
&\ge m^{\beta}\sum_{k=1}^{m} \frac{1}{k (\log(\kappa m))^\mu}
\left(1-\cos\left(\frac{2\pi k}{m}\right)\right) \\
&= \frac{m^\beta }{( \log(\kappa m))^\mu}\frac{1}{m}\sum_{k=1}^{m} \frac{1}{\frac{k}{m}}
\left(1-\cos\left(\frac{2 \pi k}{m}\right)\right).
\end{align*}
Now we observe that
\[
\lim_{m\rightarrow \infty}\frac{1}{m}\sum_{k=1}^{m} \frac{1}{\frac{k}{m}}
\left(1-\cos\left(\frac{2 \pi k}{m}\right)\right)=\int_0^1 \frac{1}{x}(1-\cos(2\pi x))\rd x>0.
\]
(Note that $x\mapsto \frac{1}{x}(1-\cos(2\pi x))$ is continuous on $[0,1]$.)
Thus we have, for $m$ sufficiently large, that
\[
m^{\beta}\left(f(0)-f\left(m^{-1}\right)\right)
\ge \frac{m^\beta}{ (\log(\kappa m))^\mu}\int_0^1 \frac{1}{x}(1-\cos(2\pi x))\rd x
\overset{m\rightarrow\infty}{\longrightarrow}\infty\,.
\]
In other words, we have that for any $L>0$ and sufficiently large $m$
\[
f(0)-f\left(\frac{1}{m}\right)>L \left(\frac{1}{m}\right)^\beta\,,
\]
so that $f$ is not H\"older continuous with H\"older coefficient $\beta$.
\end{proof}

Figure \ref{fig:log-example} shows $f$ with $\kappa=\exp({\rm e}^2)$ and $\mu=3/2$.

\begin{figure}[h]
\begin{center}
\includegraphics{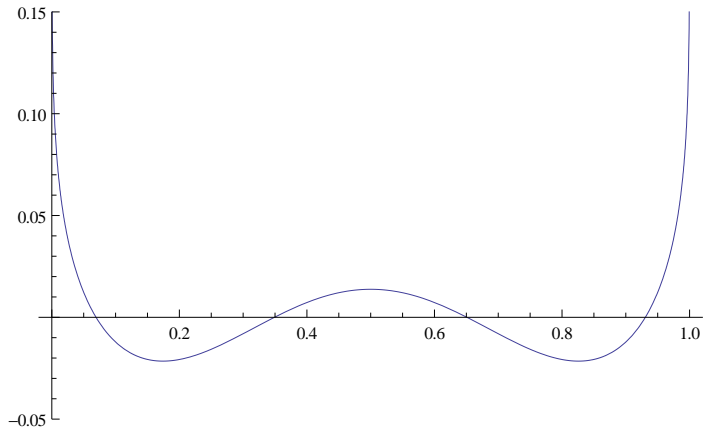}
\end{center}
\caption{A non-H\"older continuous function which lies in a $\log$-Korobov
space. }\label{fig:log-example}
\end{figure}

\subsection{The $\log$-cosine space}

To remove the periodicity assumption inherent to the $\log$-Korobov space we introduce the so-called 
$\log$-cosine space. The idea is to replace the trigonometric basis of the $\log$-Korobov space by the cosine basis
\begin{equation*}
1, \sqrt{2} \cos (\pi x), \sqrt{2} \cos (2\pi x), \sqrt{2} \cos (3\pi x), \ldots\ .
\end{equation*}
Set $\sigma_0(x) = 1$ and $\sigma_k(x) = \sqrt{2} \cos (k \pi x)$ for $k \in \mathbb{N}$. 
For vectors $\bsx \in [0,1]^s$ and $\bsk \in \mathbb{N}_0^s$ we set $\sigma_{\bsk}(\bsx) = \prod_{j=1}^s \sigma_{k_j}(x_j)$. 
The collection $(\sigma_{\bsk})_{\bsk\in\NN_0^s}$ is an orthonormal basis of $L_2([0,1]^s)$ (cf.~\cite{DNP}). 
We can define a reproducing kernel $C_{\log, \mu, \bsgamma}$ by
\begin{equation*}
C_{\log, \mu, \bsgamma}(\bsx, \bsy) =  \sum_{\bsk \in \mathbb{N}^s} r_{\mu, \bsgamma}(\bsk) \sigma_{\bsk}(\bsx) \sigma_{\bsk}(\bsy),
\end{equation*}
where $r_{\mu,\bsgamma}$ is defined as in Section~\ref{subsec_log_kor}. The corresponding reproducing kernel Hilbert space $H_{C_{\log, \mu, \bsgamma}}$ is a space of cosine series
\begin{equation*}
f(\bsx) = \sum_{\bsk \in \mathbb{N}_0^s} \widetilde{f}(\bsk) \sigma_{\bsk}(\bsx),
\end{equation*}
where the cosine coefficients are given by
\begin{equation*}
\widetilde{f}(\bsk) = \int_{[0,1]^s} f(\bsx) \sigma_{\bsk}(\bsx) \rd \bsx.
\end{equation*}
The inner product in this space is given by
\begin{equation*}
\langle f, g \rangle_{H_{C_{\log, \mu, \bsgamma}}} = 
\sum_{\boldsymbol{k} \in \mathbb{N}_0^s} \widetilde{f}(\bsk) \overline{\widetilde{g}(\bsk) } r^{-1}_{\mu, \bsgamma}(\bsk).
\end{equation*}

\subsection{The tent-transform and numerical integration in the $\log$-cosine space}

The tent-transform was first considered in the context of numerical integration in \cite{Hick02} (therein called baker's transform). 
In \cite{DNP} it was shown that numerical integration in the Korobov space 
$H_{K_{\alpha,\bsgamma}}$ using lattice rules is related to numerical integration in a certain function space based on the system 
$(\sigma_{\bsk})_{\bsk\in \NN_0^s}$ using tent-transformed lattice rules. 
We review the necessary steps to explain how this applies to  the setting considered here.

The tent-transform $\rho:[0,1] \to [0,1]$ is given by
\begin{equation*}
\rho(x) = 1 - |2 x-1|.
\end{equation*}
For vectors $\bsx \in [0,1]^s$ we apply the tent-transform component-wise, that is, we write 
$\rho(\bsx) = (\rho(x_1), \ldots, \rho(x_s))$. As in \cite[Theorem~2]{DNP} it follows that
\begin{align*}
e(H_{C_{\log, \mu, \bsgamma}}, P_{\rho}) = & -1 + \frac{1}{N^2} \sum_{n=0}^{N-1} \sum_{m=0}^{N-1} C_{\log, \mu, \bsgamma}
(\bsx_n, \bsx_m) = \sum_{\bsk \in L^\perp\setminus \{\bszero\}} r_{\mu, \bsgamma}(\bsk) =  e(H_{K_{\log, \mu, \bsgamma}}, P),
\end{align*}
where $P = \{\bsx_0, \bsx_1, \ldots, \bsx_{N-1}\}$ is a lattice point set given by \eqref{lattice_points} and where
$P_{\rho} = \{\rho(\bsx_0), \rho(\bsx_1), \ldots, \rho(\bsx_{N-1})\}$ is the corresponding tent-transformed lattice point set.

Thus all the results for the $\log$-Korobov space using lattice rules also apply to numerical integration in the 
$\log$-cosine space using tent-transformed lattice rules. For simplicity we state the results only for the 
$\log$-Korobov space and lattice rules in the following.

\section{Component-by-component construction}\label{sec_cbc}

The component-by-component construction of lattice rules was invented
independently in \cite{Kor59} and \cite{SR02}, and was developed further in
\cite{CKN06,D04,K03,KSS12,NC06,NC06b,SKJ02}. The idea is to find a good generating
vector $\bsg = (g_1,g_2, \ldots, g_s)$ 
one component at a time by minimizing the worst-case error of the $(d+1)$-dimensional lattice rule 
generated by $(g_1,\ldots,g_d, g_{d+1})$ as a function of $g_{d+1}$ in the $(d+1)$-st step, holding the components 
$g_1, g_2, \ldots, g_{d}$ fixed. In order to emphasize the dependence of the worst-case error 
$e(H_{K_{\log, \mu, \bsgamma}}, P)$ on the generating vector $\bsg$, we write $e(\bsg)$ 
for the worst-case error in $H_{K_{\log, \mu, \bsgamma}}$ using a lattice rule with generating vector $\bsg$.

\subsection{The algorithm}

\begin{algorithm}\label{algcbc}
Let $s\in\NN$ and a prime $N$ be given.
Construct $\bsg^\ast=(g_1^\ast,\ldots,g_s^\ast)\in\{1,\ldots,N-1\}^s$ as follows.
\begin{itemize}
\item Set $g_1^\ast = 1$.
\item For $d\in \{1, 2, \ldots, s-1\}$ assume that $g_1^\ast,\ldots,g_d^\ast$ have already been found. Now choose $g_{d+1}^\ast\in \{1, 2, \ldots,N-1\}$ such that
$$ e^2 (g_1^\ast,\ldots,g_{d}^\ast,g)$$
is minimized as a function of $g$.
\item Increase $d$ and repeat the second step until $(g_1^\ast,\ldots,g_s^\ast)$ is found.
\end{itemize}
\end{algorithm}

This algorithm requires one to compute the square worst-case error $ e^2 (g_1^\ast,\ldots,g_{d}^\ast,g)$ for each $g \in \{1, 2, \ldots, N-1\}$. Hence it is important to have an efficient method for computing this quantity. We describe such a method in the following subsection.

An efficient implementation, called fast CBC algorithm, using the fast Fourier transform was developed in \cite{NC06,NC06b} which applies to Korobov spaces, hence it is clear that the fast CBC algorithm also applies to the construction of good lattice rules for integration in the $\log$-Korobov space, see \cite{NC06, NC06b} for details.

\subsection{Computation of the worst-case error}

Using \eqref{wce} and the fact that $$\frac{1}{N}\sum_{n=0}^{N-1}\exp(2\pi \icomp k n/N)=\left\{
\begin{array}{ll}
1 & \mbox{ if } k \equiv 0 (\operatorname{mod}N),\\
0 & \mbox{ if } k \not\equiv 0 (\operatorname{mod}N),
\end{array}\right.$$
the squared worst-case error can be expressed as
$$e^2 (\bsg)= \sum_{\emptyset\neq\uu\subseteq \{1, \ldots, s\} }\sum_{\bsh_\uu \in \ZZ_{\ast}^{\abs{\uu}}}
r_{\mu,\bsgamma} (\bsh_\uu) \frac{1}{N} \sum_{n=0}^{N-1} \exp(2 \pi \icomp (\bsh_{\uu} \cdot \bsg_{\uu}) n/N),$$
where $\bsg_{\uu}$ is the projection of $\bsg$ onto the coordinates contained in $\uu \subseteq \{1, \ldots, s\}$, and where $\ZZ_{\ast}=\ZZ\setminus\{0\}$.
From this we further obtain
\begin{equation*}
e^2 (\bsg)= -1+\frac{1}{N}\sum_{n=0}^{N-1}
\prod_{j=1}^s \left[1+\gamma_j \sum_{h_j \in \ZZ_{\ast}} \frac{\exp(2 \pi
\icomp h_j g_jn/N)}{|h_j| (\log (\kappa |h_j|))^{\mu}}\right].
\end{equation*}
We have
\begin{align*}
\sum_{h \in \ZZ_{\ast}} \frac{\exp(2 \pi
\icomp h g n/N)}{|h| (\log (\kappa |h|))^{\mu}}
&=\sum_{h=1}^{N-1} \frac{\exp(2 \pi
\icomp h g n/N)}{|h| (\log (\kappa |h|))^{\mu}}
+\sum_{m\in \ZZ_\ast}\sum_{h=0}^{N-1} \frac{\exp(2 \pi
\icomp (h+m N) g n/N)}{|h+m N| (\log (\kappa |h+m N|))^{\mu}}\\
&=\sum_{h=0}^{N-1} \widehat{c}(h)
\exp(2 \pi \icomp h g n/N)\,,
\end{align*}
where
\begin{equation}
\widehat{c}(h):=\begin{cases}
\sum_{m\in \ZZ_\ast} |m N|^{-1} (\log (\kappa |m N|))^{-\mu} & \mbox{ if } h=0, \\
\sum_{m\in \ZZ} |h+m N|^{-1} (\log (\kappa |h+m N|))^{-\mu} & \mbox{ if } h\ne 0.
\end{cases}
\end{equation}
We do not have an explicit formula for $\widehat{c}$, not even for special cases of $\mu$ and $\kappa$. 
However, for evaluating the slowly converging series defining $\widehat{c}$ one can use the Euler-McLaurin summation formula, see \cite[page 806, item 23.1.30]{abst}. 
These values can then be stored and reused in the computation of the worst-case error. 

The Euler-McLaurin summation formula states that if $f \in C^{2m}([{z_1},{z_2}])$, then 
$$\sum_{l={z_1}}^{z_2} f(l) = \int_{z_1}^{z_2} f(t)\rd t+\tfrac{1}{2}(f({z_1})+f({z_2})) +\sum_{j=1}^m \frac{B_{2j}}{(2 j)!}(f^{(2j-1)}({z_2})-f^{(2j-1)}({z_1}))+(R(f,m,z_1,z_2),$$
and 
\[
    \left|(R(f,m,z_1,z_2)\right|\leq\frac{2 \zeta (2m)}{(2\pi)^{2m}}\int_{z_1}^{z_2}\left|f^{(2m)}(x)\right|\rd x \,
\]
where $B_i$ is the $i$-th Bernoulli number and $\zeta$ denotes the Riemann
$\zeta$-function. For example consider  $$f(x)=\frac{1}{(h+N x)(\log(\kappa(h+ N x)))^\mu}.$$
Let $g_{k,\beta}(x)=\frac{1}{(h+N x)^k(\log(\kappa(h+ N x)))^\beta}$. Then 
$f=g_{1,\mu}$ and 
\begin{align}
\nonumber g_{k,\beta}'(x)&=-\frac{N k}{(h+N x)^{k+1}\log(\kappa(h+ N x))^{\beta}}-\frac{N \beta}{(h+N x)^{k+1}\log(\kappa(h+ N x))^{\beta+1}}\\
\label{EMLdiff}&=-N \left(k g_{k+1,\beta}(x)+ \beta g_{k+1,\beta+1}(x)\right)\,.
\end{align}
Repeated application of the differentiation rule \eqref {EMLdiff}
gives
\begin{align*}
f'(x)&= -N (g_{2,\mu}(x)+\mu g_{2,\mu+1}(x))\\
f''(x)&= N^2 \big(2 g_{3,\mu}(x)+ \mu g_{3,\mu+1}(x)  
+2 \mu g_{3,\mu+1}(x)+\mu(\mu+1)g_{3,\mu+2}(x)\big)\\
f'''(x)&=- N^3 \big(6 g_{4,\mu}(x)+2\mu g_{4,\mu+1}(x)
+3 \mu  g_{4,\mu+1}(x)  + \mu(\mu+1) g_{4,\mu+2}(x)  \\
&\;\;+6 \mu g_{4,\mu+1}(x) +2 (\mu+1) \mu g_{4,\mu+2}(x)
+3\mu(\mu+1)g_{4,\mu+2}(x) +\mu(\mu+1)(\mu+2)g_{4,\mu+3}(x)
\big)
\end{align*}
and so on. We see by induction that
\begin{itemize}
\item $f^{(2k)}$ is positive on $(1,\infty)$;
\item $f^{(2k-1)}$ is negative on $(1,\infty)$;
\item $|f^{(k)}|$ is a linear combination of $2^k$ functions $g_{k+1,\mu+\ell}$, $\ell \in \{0,1,\ldots,k\}$, with
all coefficients less or equal $N^k\mu(\mu+1)\cdots(\mu+k-1)$.
\end{itemize}
In particular we can estimate
\begin{align*}
\left|f^{(k)}(x)\right|&\le (2N)^k \mu (\mu+1)\cdots(\mu+k-1)
g_{k+1,\mu}(x) 
\end{align*}
and 
\begin{align*}
\int_{z_1}^{\infty}\left|f^{(2m)}(x)\right|\rd x 
& = \int_{z_1}^{\infty} f^{(2m)}(x)\rd x \\
&= - f^{(2m-1)}(z_1)= |f^{(2m-1)}(z_1)| \\
&\le   
\frac{1}{2N}\left(\frac{2 N}{h+N z_1}\right)^{2m}\frac{ \mu (\mu+1)\cdots(\mu+2m-2)}{\log(\kappa(h+N z_1))^\mu}\,.
\end{align*}

Thus we get 
\begin{align*}
|(R(f,m,z_1,z_2)| &< \frac{\zeta (2m)}{(2\pi)^{2m}} \frac{1}{N}
\left(\frac{2N}{h+N z_1}\right)^{2m}\frac{ \mu (\mu+1)\cdots(\mu+2m-2)}{\log(\kappa(h+N z_1))^\mu}\\
&< \frac{\pi^2}{6} \frac{\mu (\mu+1)\cdots(\mu+2m-2)}{N}\left(\frac{1}{\pi z_1}\right)^{2m}
\end{align*}
uniformly for all $z_2>z_1$. Now $|(R(f,m,z_1,\infty)|$ can be made arbitrarily small by fixing $m$ and choosing $z_1$ big enough.

\section{Error bounds}\label{sec_err_bnds}

The classical approach for proving error bounds on the worst-case error in the
Korobov space using lattice rules employs an inequality by Johan Jensen, 
which asserts that 
\begin{equation}\label{Jensen1}
\bigg(\sum_k a_k \bigg)^\lambda \le \sum_k a_k^\lambda 
\end{equation}
for a sequence of nonnegative numbers $a_k$ and $0 < \lambda \le 1$; see \cite{DKS13,LP14}. However, in our setting, this inequality can only be used with 
$\lambda = 1$, for which it is trivial. To obtain a more precise inequality, roughly speaking, the idea is to consider the mapping
\begin{equation}\label{mapping_idea}
\frac{1}{z (\log (\kappa z))^\mu}  \mapsto \frac{1}{z (\log (\kappa z))^{\mu \lambda}}
\end{equation}
for $\mu \lambda > 1$.
We have the following generalization of inequality \eqref{Jensen1}:
\begin{lemma}\label{lemjensen}
Let $\phi: [0,z) \to \mathbb{R}$ be concave on $[0,z)$, where $0 < z \le \infty$, and let $0 \le a_1, a_2, \ldots$ be real numbers such that $\sum_{k=1}^n a_k < z$. Then
\[
(n-1)\phi(0)+\phi\left(\sum_{k=1}^n a_k\right)\le \sum_{k=1}^n \phi(a_k) \,.
\]

If $\sum_{k=1}^\infty a_k < z$, $\phi(0)\ge 0$ and, in addition, 
there is a $0 < z_0 \le z$ such that $\phi(z_0) \ge 0$, then 
\begin{equation}\label{Jensen2}
\phi\left(\sum_{k=1}^\infty a_k\right) \le \sum_{k=1}^\infty \phi(a_k)\,.
\end{equation}
\end{lemma}

\begin{proof}
W.l.o.g. we may assume that not all the $a_k$ are equal to 0.
Let $A:=\sum_{k=1}^n a_k$.
Since $\phi$ is concave we have
\[
\left(1-\frac{a_k}{A}\right)\phi(0)+\frac{a_k}{A}\phi(A)\le
\phi\left(\left(1-\frac{a_k}{A}\right)\cdot 0+\frac{a_k}{A}A\right)=\phi(a_k)\,.
\]
Summing over $k\in \{1,2,\ldots,n\}$ gives
\begin{equation}\label{jensen_gen}
(n-1)\phi(0)+\phi(A)
\le\sum_{k=1}^n \phi(a_k)\,.
\end{equation}
If $\phi(0) \ge 0$, inequality \eqref{jensen_gen} remains valid without the term $(n-1) \phi(0)$. Since 
$\sum_{k=1}^\infty a_k < \infty$, there exists an $n_0$ such that $0 \le a_k <
z_0$ for all $k \ge n_0$. Since $\phi$ is concave $\phi(0)\ge 0$, and
$\phi(z_0)\ge0$, we have that $\phi(x)\ge 0$ on $[0,z_0)$. Thus $\phi(a_k) \ge
0$ for all $k \ge n_0$. Therefore we have, for all $n \ge n_0$, 
\begin{equation*}
\phi\left( \sum_{k=1}^n a_k \right) \le \sum_{k=1}^n \phi(a_k) \le \sum_{k=1}^\infty \phi(a_k)\,.
\end{equation*}
Letting $n \to \infty$ yields the result.
(Note that concave functions on open intervals are continuous.)
\end{proof}

In order to be able to apply Lemma~\ref{lemjensen}, we now proceed by constructing a mapping that is concave and shares some other useful properties 
with the mapping
$z\mapsto z^{\lambda}$.

Let $\lambda$ be a fixed number in $(1/\mu,1]$. We define a function $\psi: (0,1] \To\RR$ by
\begin{equation}\label{map_psi}
\psi (z):= z (\log (1/z))^{\mu(1-\lambda)}.
\end{equation}
Furthermore, we define a function $\varphi:[0,\infty)\To\RR$ by
$$\varphi (z):=\begin{cases}
             0 &\mbox{if $z=0$,}\\
	     \psi (z) &\mbox{if $z\in (0,{\rm e}^{-2\mu}]$,}\\
	     (z-{\rm e}^{-2\mu})\psi'({\rm e}^{-2\mu})+\psi ({\rm e}^{-2\mu}) & \mbox{if $z>{\rm e}^{-2\mu}$.}
            \end{cases}
 $$
We summarize some properties of the function $\varphi$ in the following lemma.
\begin{lemma}\label{lemphi}
Let $\varphi:[0,\infty) \To \RR$ be defined as above. Then the following statements hold:
\begin{enumerate}
\renewcommand{\theenumi}{{\it (\alph{enumi})}}
\renewcommand{\labelenumi}{\theenumi}
\item \label{lemphi_inc}$\varphi$ is strictly increasing;
\item \label{lemphi_con}$\varphi$ is concave;
\item \label{lemphi_ge1}$\varphi(1)\ge 1$;
\item \label{lemphi_mul}if $x,y\in (0,{\rm e}^{-2\mu}]$, then $\varphi (xy)\le \varphi (x)\varphi  (y)$;
\item \label{lemphi_inv}
for $y\in (0, {\rm e}^{-2\mu} (2\mu)^{\mu (1-\lambda)}]$ it is true that
\begin{equation*}
\varphi^{-1} (y) \le \frac{y}{(\log (1/y))^{\mu (1-\lambda)}}\,.
\end{equation*}
\end{enumerate}
\end{lemma}

\begin{proof}
Regarding \ref{lemphi_inc}, it is easily checked that $\lim_{z\To 0} \psi (z)=0$ and that $\psi'$ is positive on $(0,{\rm e}^{-2\mu}]$. Furthermore,
$\psi'({\rm e}^{-2\mu})>0$, which implies that $\varphi$ is strictly increasing.

Regarding \ref{lemphi_con}, it is easy to see that $\psi$ is concave on $(0,{\rm e}^{-2\mu}]$ by computing the second derivative, so $\varphi$ is concave.

Regarding \ref{lemphi_ge1}, since $\psi ({\rm e}^{-2\mu})={\rm e}^{-2\mu} (2\mu)^{\mu(1-\lambda)}$, and $\psi '({\rm e}^{-2\mu})=\frac{\lambda +1}{2}(2\mu)^{\mu (1-\lambda)} $,
we obtain
\begin{eqnarray*}
 \varphi(1)&=&(1-{\rm e}^{-2\mu})\frac{\lambda +1}{2}(2\mu)^{\mu (1-\lambda)}+{\rm e}^{-2\mu} (2\mu)^{\mu(1-\lambda)}\\
&=&\frac{\lambda +1}{2} (2\mu)^{\mu (1-\lambda)} + {\rm e}^{-2\mu}(2\mu)^{\mu (1-\lambda)}\left(1-\frac{\lambda +1}{2}\right)\\
&\ge &\frac{\lambda +1}{2} (2\mu)^{\mu (1-\lambda)}.
\end{eqnarray*}
Let now $f(\lambda):=\frac{\lambda +1}{2} (2\mu)^{\mu (1-\lambda)}$ for $\lambda\in (1/\mu,1]$. Then we have
$f'(\lambda)=2^{\mu(1-\lambda)-1}\mu^{\mu(1-\lambda)}(1-\mu(1+\lambda)\log (2\mu))\le 0$,
which implies that $f$ is non-increasing on $(1/\mu,1]$.
Hence we conclude that $\varphi (1)\ge f(\lambda)\ge f(1)=1$.

Regarding \ref{lemphi_mul}, assume that $x,y\in (0,{\rm e}^{-2\mu}]$. This, in particular, means that $x,y\in (0,{\rm e}^{-2}]$ and that $\log(1/x),\log(1/y)\ge 2$.
Thus,
\begin{eqnarray*}
\varphi(xy)&=&xy(\log (1/(xy)))^{\mu(1-\lambda)}\\
&=&xy(\log (1/x)+\log (1/y))^{\mu(1-\lambda)}\\
&\le& xy(\log (1/x)\log (1/y))^{\mu(1-\lambda)}\\
&=&\varphi (x)\varphi (y).
\end{eqnarray*}

For \ref{lemphi_inv}, fix $y$ with  $0<y\le {\rm e}^{-2\mu}
(2\mu)^{\mu (1-\lambda)}=\varphi ({\rm e}^{-2\mu})$. As $\varphi$ is
non-decreasing, this implies that $z:=\varphi^{-1} (y)\in (0,{\rm
e}^{-2 \mu}]$.

Since $\varphi (z)=y$, we have
$$y=z (\log (1/z))^{\mu (1-\lambda)}\ \ \mbox{or}\ \ \frac{1}{z}=\frac{(\log (1/z))^{\mu (1-\lambda)}}{y}.$$
Hence,
\begin{eqnarray*}
\log (1/z)&=&\mu (1-\lambda)\log\log (1/z) + \log (1/y)\\
&\ge& \mu (1-\lambda) \log (2\mu) + \log (1/y)\\
&\ge& \log (1/y).
\end{eqnarray*}
This yields
$$\varphi^{-1}(y)=z=\frac{y}{(\log (1/z))^{\mu (1-\lambda)}}\le \frac{y}{(\log (1/y))^{\mu (1-\lambda)}}$$
as claimed.
\end{proof}

We also have the following lemma.
\begin{lemma}\label{lemphias}
 Let $\varphi: [0,\infty)\To\RR$ be defined as above, let $\gamma>0$, and 
$\kappa\ge\max\{\exp(\mathrm{e}^2),\exp(\mathrm{e}^2 \gamma^{1/\mu})\}$. 

Then for all $k \neq 0$ we have
\begin{equation*}
\varphi\left(r_{\mu,\gamma}(k)\right) \le  \gamma \, D_{\mu,\lambda,\kappa} 
\max\{1,\log \gamma^{-1}\}^{\mu(1-\lambda)}  \frac{1}{\abs{k} (\log (\kappa \abs{k}))^{\mu \lambda}},
\end{equation*}
where $D_{\mu,\lambda,\kappa}>0$ is a constant depending only on $\mu,\lambda$, and $\kappa$.
\end{lemma}

\begin{proof}
Note that by the choice of $\kappa$ we have $r_{\mu,\gamma} (k)\le {\rm e}^{-2\mu}$ and $\log (\kappa \abs{k})\ge 1$ 
for any $k\neq 0$.
We then have
\begin{align*}
\varphi (r_{\mu,\gamma}(k))&=\varphi\left(\frac{\gamma}{\abs{k} (\log (\kappa \abs{k}))^\mu}\right)\\
&=\frac{\gamma}{\abs{k} (\log (\kappa \abs{k}))^\mu}\left(\log ( \gamma^{-1} \abs{k} (\log (\kappa \abs{k}))^\mu)\right)^{\mu (1-\lambda)} \\
&=\frac{\gamma}{\abs{k} (\log (\kappa \abs{k}))^\mu}\left(\log \gamma^{-1} + \log \abs{k} + 
\mu \log\log (\kappa \abs{k})\right)^{\mu (1-\lambda)}.
\end{align*}
Furthermore,
\begin{eqnarray*}
 \lefteqn{\log \gamma^{-1} + \log \abs{k} + \mu \log\log (\kappa \abs{k}))}\\
&\le& \max\{1,\log \gamma^{-1}\} + \log \abs{k} + \mu \log\log (\kappa \abs{k})\\
&\le& \mu \max\{1,\log \gamma^{-1}\} \left(1+\log \abs{k} + \log\log (\kappa \abs{k})\right)\\
&\le& \mu \max\{1,\log \gamma^{-1}\} \left(1+\log (\kappa \abs{k}) + \log\log (\kappa \abs{k})\right)\\
&\le& 3 \mu \max\{1,\log \gamma^{-1}\} \log (\kappa \abs{k}),
\end{eqnarray*}
where we used $\log (\kappa \abs{k})\ge 1$. This yields the result.

\end{proof}

Now the following theorem shows how the worst-case error using a lattice rule with a
generating vector constructed by Algorithm~\ref{algcbc} can be bounded.
\begin{theorem}\label{thmcbc}
 Let $\bsg^\ast=(g_1^\ast,\ldots,g_s^\ast)\in\{1,\ldots,N-1\}^s$ be a vector constructed by Algorithm~\ref{algcbc}. Let 
$\bsgamma = (\gamma_j)_{j \in \mathbb{N}}$ be a sequence of positive real numbers and assume that 
$\kappa \ge \sup_{j \in \mathbb{N}} \max\{\exp(\mathrm{e}^2),\exp(\mathrm{e}^2 \gamma_j^{1/\mu})\}$. Then for every $d\in \{1, \ldots, s\}$ and any $1/\mu < \lambda \le 1$ 
it is true that
\begin{align}\label{eqcbc}
e^2 (g_1^\ast,\ldots,g_d^\ast)\le &  
\frac{T_d (\mu,\lambda,\kappa,\bsgamma)}{N \left(\log (N/T_d (\mu,\lambda,\kappa,\bsgamma))\right)^{\mu(1-\lambda)}}
\end{align}
for any $N\ge  \mathrm{e}^{2\mu} (2\mu)^{\mu (\lambda-1)}T_d (\mu,\lambda,\kappa,\bsgamma)$, where
\begin{equation*}
T_d(\mu, \lambda,\kappa,\bsgamma) =  3 \prod_{j=1}^d 
\left(1+2C_{\mu,\lambda,\kappa}\ \gamma_{j} \max\{1,\log \gamma_{j}^{-1}\}^{\mu(1-\lambda)}\right), 
\end{equation*}
with a constant $C_{\mu,\lambda,\kappa}$ only depending on $\mu$, $\lambda$ and $\kappa$.
\end{theorem}

\begin{proof}
Using induction on $d$ we first show that
\begin{equation}\label{e_bound}
\varphi(e^2 (g_1^\ast, \ldots, g_{d}^\ast))\le\frac{T_d (\mu,\lambda,\kappa,\bsgamma)}{N}.
\end{equation}
The result then follows from Lemma~\ref{lemphi}\ref{lemphi_inv}.

Note that if $\kappa \ge \sup_{j \in \mathbb{N}} \max\{\exp(\mathrm{e}^2),\exp(\mathrm{e}^2 \gamma_j^{1/\mu})\}$, 
then $r_{\mu, \gamma_j}(k_j) \le \mathrm{e}^{-2\mu}$ if $k_j\neq 0$. We can then use \ref{lemphi_ge1} and \ref{lemphi_mul} of Lemma~\ref{lemphi} to obtain
\begin{equation*}
\varphi(r_{\mu, \bsgamma}(\bsk) ) = \varphi \left( \prod_{\substack{ j=1 \\ k_j \neq 0 }}^s r_{\mu, \gamma_j}(k_j) \right) 
\le \prod_{\substack{j=1 \\ k_j \neq 0}}^s \varphi(r_{\mu, \gamma_j}(k_j)) \le \prod_{j=1}^s \varphi(r_{\mu, \gamma_j}(k_j)).
\end{equation*}

 For $d=1$, we have $g^\ast_1=1$. Using Lemma~\ref{lemjensen} and Lemma~\ref{lemphi}\ref{lemphi_inc}, we obtain
\begin{eqnarray*}
 \varphi(e^2 (1))&=&\varphi \left(\sum_{\substack{k\in\ZZ\setminus\{0\}\\ k\equiv 0 \pmod{N}}} r_{\mu,\gamma_1} (k)\right)\\
&\le & \sum_{\substack{k\in\ZZ\setminus\{0\}\\ k\equiv 0 \pmod{N}}}\varphi(r_{\mu,\gamma_1} (k))\\
&\le & 2 D_{\mu,\lambda,\kappa} \gamma_1 \max\{1,\log \gamma_1^{-1}\}^{\mu(1-\lambda)} 
\sum_{\substack{k=1\\ k\equiv 0 \pmod{N}}}^\infty \frac{1}{k(\log (\kappa k))^{\mu\lambda}}\\
&= &  2 D_{\mu,\lambda,\kappa} \gamma_1 \max\{1,\log\gamma_1^{-1}\}^{\mu(1-\lambda)} \sum_{k=1}^\infty \frac{1}{kN(\log (\kappa kN))^{\mu\lambda}}\\
&\le&\frac{2}{N} D_{\mu,\lambda,\kappa} \gamma_1 \max\{1,\log \gamma_1^{-1}\}^{\mu(1-\lambda)} \sum_{k=1}^\infty \frac{1}{k(\log (\kappa k))^{\mu\lambda}}\\
&=&\frac{2}{N} C_{\mu,\lambda,\kappa} \gamma_1 \max\{1,\log \gamma_1^{-1}\}^{\mu(1-\lambda)},
\end{eqnarray*}
where $C_{\mu,\lambda,\kappa} = D_{\mu, \lambda, \kappa} \sum_{k=1}^\infty \frac{1}{k (\log (\kappa k))^{\mu \lambda}} < \infty$ and where $D_{\mu, \lambda, \kappa}$ is defined as in Lemma~\ref{lemphias}. This implies the result for $d=1$.

Suppose now the result has already been shown for dimension $d<s$, i.e., there exists a vector $\bsg_d^\ast=(g_1^\ast,\ldots,g_d^\ast)$ satisfying~\eqref{eqcbc}.
As the algorithm chooses $g_{d+1}^\ast$ to minimize $e^2 (\bsg_d^\ast,g)$, we clearly have $e^2 (\bsg_d^\ast,g_{d+1}^\ast) \le e^2 (\bsg_d^\ast,g)$ for all $g \in \{1,2,\ldots,N-1\}$. From the monotonicity of $\varphi$, see Lemma~\ref{lemphi}\ref{lemphi_inc}, we obtain 
$$
\varphi(e^2 (\bsg_d^\ast,g_{d+1}^\ast)) \le \frac{1}{N-1}\sum_{g\in\{1,\ldots,N-1\}} \varphi(e^2 (\bsg_d^\ast,g))\,.
$$
We therefore have
\begin{eqnarray*}
\lefteqn{ \varphi(e^2 (\bsg_d^\ast,g_{d+1}^\ast))}\\
&\le& \frac{1}{N-1}\sum_{g\in\{1,\ldots,N-1\}}
\varphi\left( \sum_{\substack{(\bsk_d,k)\in\ZZ^{d+1}\setminus\{\bszero\}\\ (\bsg_d^\ast,g)\cdot(\bsk_d,k)\equiv 0\pmod{N}}}\hspace{-1em}r_{\mu,\bsgamma}(\bsk_d,k)\right)\\
&=&\frac{1}{N-1}\sum_{g\in\{1,\ldots,N-1\}}\varphi\left(\sum_{\substack{\bsk_d\in\ZZ^{d}\setminus\{\bszero\}\\ \bsg_d^\ast\cdot\bsk_d\equiv 0\pmod{N}}}\hspace{-1em}r_{\mu,\bsgamma}(\bsk_d)
+\sum_{\bsk_d\in\ZZ^d}\sum_{\substack{k\in\ZZ\setminus\{0\}\\ (\bsg_d^\ast,g)\cdot(\bsk_d,k)\equiv 0\pmod{N} }}\hspace{-2em}r_{\mu,\bsgamma}(\bsk_d,k)\right)\\
&\le&\frac{1}{N-1}\!\sum_{g\in\{1,\ldots,N-1\}}\!\left(\varphi\left(\sum_{\substack{\bsk_d\in\ZZ^{d}\setminus\{\bszero\}\\ \bsg_d^\ast\cdot\bsk_d\equiv 0\pmod{N}}}\hspace{-1em}r_{\mu,\bsgamma}(\bsk_d)\right)
+\varphi\left(\!\sum_{\bsk_d\in\ZZ^d}\!\!\sum_{\substack{k\in\ZZ\setminus\{0\}\\ (\bsg_d^\ast,g)\cdot(\bsk_d,k)\equiv 0\pmod{N} }}\hspace{-2em}r_{\mu,\bsgamma}(\bsk_d,k)\right)\!\!\right),\\
\end{eqnarray*}
where we used Lemma~\ref{lemjensen} to obtain the last inequality. Hence we see that
\begin{equation}\label{eqstep}
\varphi(e^2 (\bsg_d^\ast,g_{d+1}^\ast))\le \varphi (e^2 (\bsg_d^\ast)) + \theta_d,
\end{equation}
where
$$\theta_d:=\frac{1}{N-1}\sum_{g\in\{1,\ldots,N-1\}}
\varphi\left(\sum_{\bsk_d\in\ZZ^d}\sum_{\substack{k\in\ZZ\setminus\{0\}\\ (\bsg_d^\ast,g)\cdot(\bsk_d,k)\equiv 0\pmod{N} }}
r_{\mu,\bsgamma}(\bsk_d,k)\right).$$
Using Lemma~\ref{lemjensen} once again, we get
\begin{eqnarray*}
 \theta_d &\le & \sum_{\bsk_d\in\ZZ^d}\sum_{k\in\ZZ\setminus\{0\}} \frac{1}{N-1}\sum_{\substack{g\in\{1,\ldots,N-1\}\\(\bsg_d^\ast,g)\cdot(\bsk_d,k)\equiv 0\pmod{N}}}
\varphi (r_{\mu,\bsgamma}(\bsk_d,k))\\
&=&\sum_{\bsk_d\in\ZZ^d}\sum_{\substack{k\in\ZZ\setminus\{0\}\\ k\equiv 0\pmod{N} }} \frac{1}{N-1}\sum_{\substack{g\in\{1,\ldots,N-1\}\\\bsg_d^\ast\cdot\bsk_d\equiv 0\pmod{N}}}
\varphi (r_{\mu,\bsgamma}(\bsk_d,k))\\
&&+\sum_{\bsk_d\in\ZZ^d}\sum_{\substack{k\in\ZZ\setminus\{0\}\\ k\not\equiv 0\pmod{N} }} \frac{1}{N-1}\sum_{\substack{g\in\{1,\ldots,N-1\}\\(\bsg_d^\ast,g)\cdot(\bsk_d,k)\equiv 0\pmod{N}}}
\varphi (r_{\mu,\bsgamma}(\bsk_d,k))\\
& \le &\sum_{\bsk_d\in\ZZ^d}\sum_{\substack{k\in\ZZ\setminus\{0\}\\ k\equiv 0\pmod{N} }}  \varphi (r_{\mu,\bsgamma}(\bsk_d,k))
+\frac{1}{N-1}\sum_{\bsk_d\in\ZZ^d}\sum_{\substack{k\in\ZZ\setminus\{0\}\\ k\not\equiv 0\pmod{N} }}  \varphi (r_{\mu,\bsgamma}(\bsk_d,k)),
\end{eqnarray*}
where we used that there is at most one solution $g\in\{1,\ldots,N-1\}$ to the congruence $(\bsg_d^\ast,g)\cdot(\bsk_d,k)\equiv 0\pmod{N}$.

Let now
$$\Sigma_1:= \sum_{\bsk_d\in\ZZ^d}\sum_{\substack{k\in\ZZ\setminus\{0\}\\ k\equiv 0\pmod{N} }}  \varphi (r_{\mu,\bsgamma}(\bsk_d,k))$$
and
$$\Sigma_2:=\frac{1}{N-1}\sum_{\bsk_d\in\ZZ^d}\sum_{\substack{k\in\ZZ\setminus\{0\}\\ k\not\equiv 0\pmod{N} }}  \varphi (r_{\mu,\bsgamma}(\bsk_d,k)).$$ 

Regarding $\Sigma_1$, we have 
\begin{eqnarray*}
 \Sigma_1 &=&\sum_{\substack{k\in\ZZ\setminus\{0\}\\ k\equiv 0\pmod{N} }}  \varphi (r_{\mu,\gamma_{d+1}}(k))+
  \sum_{\bsk_d\in\ZZ^d\setminus\{\bszero\}}\sum_{\substack{k\in\ZZ\setminus\{0\}\\ k\equiv 0\pmod{N} }}  \varphi (r_{\mu,\bsgamma}(\bsk_d,k))\\
&=&\sum_{\substack{k\in\ZZ\setminus\{0\}\\ k\equiv 0\pmod{N} }}  \varphi (r_{\mu,\gamma_{d+1}}(k))\\
&&+ \left(\sum_{\emptyset\neq\uu \subseteq [d]} 2^{|\uu|} \prod_{j\in \uu} \left(\sum_{k=1}^{\infty} \varphi(r_{\mu,\gamma_j}(k))\right)\right)
\sum_{\substack{k\in\ZZ\setminus\{0\}\\ k\equiv 0\pmod{N} }}  \varphi (r_{\mu,\gamma_{d+1}}(k))\\
&=&\underbrace{\sum_{\substack{k\in\ZZ\setminus\{0\}\\ k\equiv 0\pmod{N} }}  \varphi (r_{\mu,\gamma_{d+1}}(k))}_{=:\Sigma_{1,1}}
\underbrace{\left(1+\sum_{\emptyset\neq\uu \subseteq [d]} 2^{|\uu|} \prod_{j\in \uu} \left(\sum_{k=1}^{\infty} \varphi(r_{\mu,\gamma_j}(k))\right)\right)}_{=:\Sigma_{1,2}}.
\end{eqnarray*}
In exactly the same way as for $d=1$, we see that
$$\Sigma_{1,1}\le \frac{2}{N}C_{\mu,\lambda,\kappa} \gamma_{d+1} \max\{1,\log \gamma_{d+1}^{-1}\}^{\mu(1-\lambda)}.$$
Moreover, 
\begin{align*}
\Sigma_{1,2} = \prod_{j=1}^d \left(1 + 2 \sum_{k=1}^\infty \varphi (r_{\mu, \gamma_j}(k) ) \right).
\end{align*}
The sum $ \sum_{k=1}^\infty \varphi (r_{\mu, \gamma_j}(k) )$ can be estimated similarly to the case for $d=1$ by removing the assumption $k\equiv 0 \pmod{N}$. Thus
$\Sigma_{1,2}\le \tfrac{1}{3} T_d (\mu,\lambda,\kappa,\bsgamma)$ and therefore
$$\Sigma_1\le \frac{2}{N}C_{\mu,\lambda,\kappa} \gamma_{d+1} \max\{1,\log \gamma_{d+1}^{-1}\}^{\mu(1-\lambda)} \frac{1}{3} T_d (\mu,\lambda,\kappa,\bsgamma).$$

For $\Sigma_2$, we have, in the same way as above
\begin{eqnarray*}
 \Sigma_2 &\le &\frac{2}{N} \sum_{\substack{k\in\ZZ\setminus\{0\}\\ k\not\equiv 0\pmod{N} }} \varphi (r_{\mu,\gamma_{d+1}}(k))\sum_{\bsk_d\in\ZZ^d}
\varphi (r_{\mu,\bsgamma}(\bsk_d))\\
&\le&\frac{2}{N}\sum_{k\in\ZZ\setminus\{0\}} \varphi (r_{\mu,\gamma_{d+1}}(k))
\left(1+\sum_{\emptyset\neq\uu \subseteq [d]} 2^{|\uu|} \prod_{j\in \uu} \left(\sum_{k=1}^{\infty} \varphi(r_{\mu,\gamma_j}(k))\right)\right)\\
&\le &\frac{2}{N}2 C_{\mu,\lambda,\kappa} \gamma_{d+1} \max\{1,\log \gamma_{d+1}^{-1}\}^{\mu(1-\lambda)} \frac{1}{3} T_d (\mu,\lambda,\kappa,\bsgamma).
\end{eqnarray*}
Summing up, we obtain
$$\theta_d \le \frac{1}{N}2C_{\mu,\lambda,\kappa} \gamma_{d+1} \max\{1,\log \gamma_{d+1}^{-1}\}^{\mu(1-\lambda)} T_d (\mu,\lambda,\kappa,\bsgamma).$$
Using this estimate and the induction hypothesis together with \eqref{eqstep}, yields
$$\varphi(e^2 ((\bsg_d^\ast,g_{d+1}^\ast)))\le\frac{1}{N} T_d (\mu,\lambda,\kappa,\bsgamma)\left(1+2C_{\mu,\lambda,\kappa} \gamma_{d+1} 
\max\{1,\log\gamma_{d+1}^{-1}\}^{\mu(1-\lambda)}\right).$$
The result follows.
\end{proof}

Theorem~\ref{thmcbc} is similar to \cite[Corollary~2]{K03} (where $\beta_j =1$). Our result gives a more refined estimate for the endpoint where $\lambda$ is arbitrarily close to $1/\mu$.

We discuss now a relation between the worst-case error integration errors in the Korobov space and $\log$-Korobov space, respectively. The worst-case error in the Korobov space $H_{K_{\alpha,\bsgamma}}$ depends on $\alpha$ and is defined for $\alpha > 1/2$ only. One can introduce a new figure-of merit by setting $\alpha=1/2$ and truncating the infinite sum in \eqref{Korobov_kernel} to $\bsk \in C(N) = (-N/2, N/2]^s \cap \mathbb{Z}^s$. This figure-of merit is then given by the finite sum
\begin{equation*}
R(\bsg) = \sum_{\bsk \in (L^\perp \cap C(N) ) \setminus \{\bszero\}} w_{1/2,\bsgamma}(\bsk)
\end{equation*}
and can be used to obtain bounds on $e(H_{K_\alpha, \bsgamma}, P)$ for any $\alpha > 1/2$ (see \cite[Theorem~5.5]{N92} or \cite[Lemma~4.20]{LP14} for the 
case where $\gamma_1=\ldots=\gamma_s = 1$). The advantage is that $R$ is independent of $\alpha$ and 
that lattice points $\bsg$ with ``small'' value for $R(\bsg)$ automatically yield a ``small'' worst-case error in   
$H_{K_\alpha, \bsgamma}$ for every $\alpha>1/2$. The worst-case error in the $\log$-Korobov space can be used in a 
similar way as the criterion $R$. Theorem~\ref{thm_ineq} below is similar to a weighted 
version of \cite[Theorem~5.5]{N92} or \cite[Lemma~4.20]{LP14} and the well-known inequality
\begin{equation*}
e^2(H_{K_{\alpha, \bsgamma} }, P) = \sum_{\bsk \in L^\perp\setminus \{\bszero\}} w_{\alpha, \bsgamma}(\bsk) 
\le \left(\sum_{\bsk \in L^\perp\setminus \{\bszero\}} w_{\alpha\lambda, \bsgamma^\lambda}(\bsk) \right)^{1/\lambda} = 
\left(e^2(H_{K_{\alpha \lambda, \bsgamma^\lambda}}, P) \right)^{1/\lambda}.
\end{equation*}
for $1/(2\alpha ) < \lambda \le 1$, where $\bsgamma^\lambda = (\gamma_j^\lambda)_{j \in \mathbb{N}}$, which is a direct consequence of \eqref{Jensen1}.

\begin{theorem}\label{thm_ineq}
For any $\bsg \in \{1, 2, \ldots, N-1\}^s$, $\mu > 1$, $\alpha > 1/2$ and $1/(2 \alpha) < \lambda \le 1$ we have
\begin{equation*}
e(H_{K_{\alpha, \bsgamma}}, P) \le   \left( e(H_{K_{\log, \mu, \bsgamma^\lambda \tau_{\mu, \alpha, \kappa, \lambda}}}, P) \right)^{1/\lambda},
\end{equation*}
where 
\begin{equation}\label{def_tau}
\tau_{\mu, \alpha, \kappa, \lambda} = \max\left\{ (\log \kappa)^\mu, \left( \frac{\mu}{2\lambda \alpha - 1} \right)^\mu \right\}, 
\end{equation}
and where $\bsgamma^\lambda \tau_{\mu, \alpha, \kappa, \lambda} = (\gamma_j^\lambda \tau_{\mu, \alpha, \kappa, \lambda})_{j \in \mathbb{N}}$.
\end{theorem}

\begin{proof}
For $1/(2 \alpha ) < \lambda \le 1$ we have
\begin{equation*}
(e^2(H_{K_{\alpha,\bsgamma}}, P))^\lambda =  \left( \sum_{\bsk \in L^\perp \setminus \{\bszero\}} w_{\alpha, \bsgamma}(\bsk) \right)^\lambda \le \sum_{\bsk \in L^\perp\setminus \{\bszero\}} w^\lambda_{\alpha, \bsgamma}(\bsk) = \sum_{\bsk \in L^\perp \setminus \{\bszero\}} \prod_{\substack{j=1 \\ k_j \neq 0}}^s \frac{\gamma_j^\lambda}{|k_j|^{2\lambda \alpha}}.
\end{equation*}
We claim that for any $k_j \in \mathbb{Z} \setminus \{0\}$ we have
\begin{equation*}
\frac{1}{|k_j|^{2\lambda \alpha}} \le \frac{\tau_{\mu, \alpha, \kappa, \lambda}}{|k_j| (\log (\kappa |k_j|))^\mu},
\end{equation*}
where 
\begin{equation*}
\tau_{\mu, \alpha, \kappa, \lambda} =  \max\left\{(\log \kappa)^\mu,  \left(\frac{\mu}{2\lambda \alpha -1 } \right)^\mu \right\}.
\end{equation*}
This amounts to showing that $\log (\kappa x) \le \tau_{\mu, \alpha, \kappa, \lambda}^{1/\mu} x^{\frac{2\lambda \alpha-1}{\mu}}$ for all $x \ge 1$. 
The result clearly holds for $x=1$. By differentiating both sides with respect to $x$, we require that 
$$x^{-1} \le \tau_{\mu, \alpha, \kappa, \lambda}^{1/\mu} \frac{2\lambda \alpha-1}{\mu} x^{\frac{2\lambda \alpha-1}{\mu} - 1}.$$ 
This means that we require 
$$\tau_{\mu, \alpha, \kappa, \lambda} \ge \left(\frac{\mu}{2\lambda \alpha-1}\right)^\mu x^{-(2\lambda \alpha-1)}$$ 
for all $x \ge 1$. Since $2 \lambda \alpha -1 > 0$, the claim follows.

Thus we have
\begin{align*}
(e^2(H_{K_{\alpha, \bsgamma}}, P))^\lambda \le & 
\sum_{\bsk \in L^\perp \setminus \{\bszero\}} \prod_{\substack{j=1 \\ k_j \neq 0}}^s 
\frac{\gamma_j^\lambda \tau_{\mu, \alpha, \kappa, \lambda} }{|k_j| (\log (\kappa |k_j|))^\mu} \\  
= & \sum_{\bsk \in L^\perp \setminus \{\bszero\}} r_{\mu, \bsgamma^\lambda \tau_{\mu, \alpha, \kappa, \lambda}}(\bsk) \\  
= & e^2(H_{K_{\log, \mu, \bsgamma^\lambda \tau_{\mu, \alpha, \kappa, \lambda}}}, P).
\end{align*}
\end{proof}

Combining Theorems~\ref{thmcbc} and \ref{thm_ineq} we obtain the following result.
\begin{corollary}
Let $\mu>1$, let $\bsgamma = (\gamma_j)_{j \in \mathbb{N}}$ be a sequence of positive real numbers and assume that 
$\kappa \ge \sup_{j \in \mathbb{N}} \exp(\mathrm{e}^2 \gamma_j^{1/\mu})$. Let 
$\bsg^\ast$ be constructed using Algorithm~\ref{algcbc} based on the $\log$-Korobov space 
$H_{K_{\log, \mu, \bsgamma}}$ and let $P$ denote the lattice rule with generating vector $\bsg^\ast$. Then for all 
$\alpha > 1/2$, $1/(2\alpha) < \lambda \le 1$ and $1/\mu < \lambda' \le 1$ we have
\begin{equation*}
e(H_{K_{\alpha, \bsgamma^{1/\lambda} / \tau_{\mu, \alpha, \kappa, \lambda}^{1/\lambda}  }}, P) 
\le \left( \frac{T_d (\mu,\lambda',\kappa,\bsgamma)}
{N \left(\log (N/T_d (\mu,\lambda',\kappa,\bsgamma))\right)^{\mu(1-\lambda')}} \right)^{1/(2\lambda )},
\end{equation*}
where $\tau_{\mu, \alpha, \kappa, \lambda}$ is given by \eqref{def_tau}, and where 
$\bsgamma^{1/\lambda} / \tau_{\mu, \alpha, \kappa, \lambda}^{1/\lambda}= 
(\gamma_j^{1/\lambda} / \tau_{\mu, \alpha, \kappa, \lambda}^{1/\lambda})_{j \in \mathbb{N}}$.
\end{corollary}

\section{Tractability}\label{sec_tract}

We now investigate the dependence of the right-hand side in \eqref{eqcbc} on the dimension. 
In order to get an upper bound on the error of numerical integration in the Korobov space $H_{K_{\alpha,\bsgamma}}$ using
lattice rules,  which is independent of the dimension, we require that
$\sum_{j=1}^\infty \gamma_j^\tau < \infty$ for some $1/(2\alpha) < \tau \le 1$. 
Theorem~\ref{thmcbc} allows us to give a more precise condition for the case where $\tau$ is close to $1/(2\alpha)$. 

To make our study more precise, we need to recall the concept of tractability. Here, we only give the definitions relevant to our setting, 
for much more detailed information we refer, e.g., to~\cite{NW08, NW10}. Let $e(N,s)$ be the $N$th minimal QMC worst-case error of integration in 
a Hilbert space $H$ given by
\begin{equation*}
e(N,s) = \inf_{P} e(H, P),
\end{equation*}
where the infimum is extended over all $N$-element point sets $P$ in $[0,1]^s$. 
We also define the initial error $e(0,s)$ as the integration error when approximating the integral by $0$, that is,
\begin{equation*}
e(0,s) = \sup_{f \in H, \|f\| \le 1} \left| \int_{[0,1]^s} f(\bsx) \rd \bsx \right|.
\end{equation*}
This is used as a reference value.

We are interested in the dependence of the $N$th minimal worst-case error on the dimension $s$. 
We consider the QMC information complexity, which is defined by
\begin{equation*}
N_{\min}(\varepsilon, s) = \min\{ N \in \mathbb{N}: e(N,s) \le \varepsilon e(0,s)\}.
\end{equation*}
This means that $N_{\min}(\varepsilon, s)$ is the minimal number of points
which are required to reduce the initial error by a factor of $\varepsilon$.
We can now define the following notions of tractability. 

We say that the
integration problem in $H$ is 
\begin{enumerate}
\item weakly QMC tractable, if
\begin{equation*}
\lim_{s + \varepsilon^{-1} \to \infty} \frac{\log N_{\min}(\varepsilon, s)}{s+\varepsilon^{-1}}  = 0;
\end{equation*}

\item polynomially QMC-tractable, if there exist non-negative numbers $c$, $p$ and $q$ such that
\begin{equation}\label{ineq_Nmin}
N_{\min}(\varepsilon, s) \le c s^q \varepsilon^{-p}.
\end{equation}

\item strongly polynomially QMC-tractable, if \eqref{ineq_Nmin} holds with $q = 0$. 
\end{enumerate}

It is well-known and easy to show that the initial error for integration in $H_{K_{\alpha,\bsgamma}}$ 
and in $H_{K_{\log, \mu, \bsgamma}}$ is equal to $1$ in both cases.

The following theorem states necessary and sufficient conditions for the
different  notions of tractability.  It turns out that these conditions are exactly
the same as the standard tractability results for the Korobov space
$H_{K_{\alpha,\bsgamma}}$, which can be found, e.g., in~\cite{NW10}.
\begin{theorem}\label{proptract}
\begin{itemize}
\item A necessary and sufficient condition for strong polynomial tractability of integration in $H_{K_{\alpha,\bsgamma}}$ 
and $H_{K_{\log, \mu, \bsgamma}}$ is
 $$\sum_{j=1}^\infty \gamma_j <\infty.$$
\item A necessary and sufficient condition for polynomial tractability of integration in $H_{K_{\alpha,\bsgamma}}$ 
and $H_{K_{\log, \mu, \bsgamma}}$ is
 $$\limsup_{s\To\infty} \frac{\sum_{j=1}^s \gamma_j}{\log s}<\infty.$$
\item A necessary and sufficient condition for weak tractability of integration in $H_{K_{\alpha,\bsgamma}}$
and $H_{K_{\log, \mu, \bsgamma}}$ is
 $$\lim_{s\To\infty} \frac{\sum_{j=1}^s \gamma_j}{s}=0.$$
\end{itemize}
\end{theorem}

\begin{proof}
We start with showing the sufficiency of the conditions stated in the theorem. Note that the proof of Theorem~\ref{thmcbc} yields, for the special choice of $\lambda=1$,
\begin{equation}\label{eqtract}
e^2(\bsg^\ast)\le \varphi^{-1}\left(
\frac{3}{N}\prod_{j=1}^s \left(1+2C_{\mu,\lambda,\kappa}\ \gamma_{j}\right)\right).
\end{equation}

Suppose first that $\sum_{j=1}^\infty \gamma_j <\infty.$ Then
\begin{equation}\label{ineq_prod_sum}
\prod_{j=1}^s (1 + 2 C_{\mu, \lambda, \kappa} \gamma_j) = \exp\left(\sum_{j=1}^s \log (1 + 2 C_{\mu, \lambda, \kappa} \gamma_j )\right) 
\le \exp\left(2 C_{\mu, \lambda, \kappa} \sum_{j=1}^{\infty} \gamma_j \right),
\end{equation}
where we used $\log (1+x) \le x$ for $x \ge 0$. This implies strong polynomial tractability.

Suppose now that $\limsup_{s\To\infty} \tfrac{1}{\log s}\sum_{j=1}^s \gamma_j<\infty.$ Then by similar arguments as for strong polynomial tractability, 
\begin{eqnarray*}
\prod_{j=1}^s \left(1+2C_{\mu,\lambda,\kappa}\gamma_j\right)\le s^{\eta}
\end{eqnarray*}
for some positive constant $\eta> 0$ that depends on $\mu,\lambda$, and $\kappa$.
Now choose $N\ge 3s^\eta {\rm e}^{2\mu}$. We then obtain, by using~\eqref{eqtract} and
applying Lemma \ref{lemphi}\ref{lemphi_inv} with $\lambda=1$,
$$e^2(\bsg^\ast)\le \frac{3s^\eta}{N}.$$
Hence, in order to achieve a worst-case error of at most $\varepsilon$, it is sufficient to choose $N$ such that
$$N\ge 3s^\eta \max\left\{{\rm e}^{2\mu},\varepsilon^{-2}\right\},$$
which implies polynomial tractability.

Finally, assume that $\lim_{s\To\infty} \tfrac{1}{s}\sum_{j=1}^s \gamma_j=0$. 
Choose $N\ge 3\exp(2C_{\mu,\lambda,\kappa}\sum_{j=1}^s \gamma_j){\rm e}^{2\mu}$.
We again apply~\eqref{eqtract} and Lemma \ref{lemphi}\ref{lemphi_inv} with $\lambda=1$ to obtain
$$e^2(\bsg^\ast)\le \frac{3 \exp(2C_{\mu,\lambda,\kappa}\sum_{j=1}^s \gamma_j)}{N}.$$
If we want to achieve a worst-case error of at most $\varepsilon$, we can choose $N$ such that
$$N\ge 3\exp \left(2C_{\mu,\lambda,\kappa}\sum_{j=1}^s \gamma_j\right)  \max\{{\rm e}^{2\mu},\varepsilon^{-2}\}.$$
Hence by the assumption that $\lim_{s\To\infty} \tfrac{1}{s}\sum_{j=1}^s \gamma_j=0$, we see that we indeed have weak tractability.

\medskip

Regarding necessity of the conditions, we consider the Korobov space $H_{K_{\alpha, \bsgamma}}$ with $\alpha=\mu/2$. 
For $\tau > 0$  sufficiently small we have $w_{\alpha, \gamma_j \tau}(k) \le r_{\mu, \gamma_j}(k)$ for all $k \in \mathbb{Z}$. 
Note that $\tau$ can be chosen independently of $\gamma_j$. Hence
\begin{equation*}
\|f\|^2_{H_{K_{\alpha, \tau \bsgamma}}} = 
\sum_{\bsk \in \mathbb{Z}^s} w^{-1}_{\alpha, \tau \bsgamma}(\bsk) |\widehat{f}(\bsk)|^2 \ge \sum_{\bsk \in \mathbb{Z}^s} r^{-1}_{\mu, \bsgamma}(\bsk) |\widehat{f}(\bsk)|^2 = 
\|f\|^2_{H_{K_{\log, \mu, \bsgamma}}},
\end{equation*}
and thus the unit ball of the Korobov space $H_{K_{\alpha, \tau \bsgamma}}$ is contained in the unit ball of the $\log$-Korobov space $H_{K_{\log, \mu, \bsgamma}}$. 
Consequently, the worst-case error of integration in the $\log$-Korobov space $H_{K_{\log, \mu, \bsgamma}}$ is at least as large as the worst-case error of integration in the 
Korobov space $H_{K_{\alpha, \bsgamma}}$. The necessary condition on $\bsgamma$ for achieving weak tractability in the Korobov space 
$H_{K_{\alpha, \bsgamma}}$ is $\lim_{s\To\infty} \tfrac{1}{s}\sum_{j=1}^s \tau\gamma_j=0$ (see, e.g.,~\cite[Theorem~16.5]{NW10}), which is equivalent to the necessary condition
on $\bsgamma$ stated in the theorem. For polynomial and strong polynomial
tractability we can proceed in the same way, 
again using~\cite[Theorem~16.5]{NW10}.
\end{proof}

Theorem~\ref{thmcbc} also allows us to obtain a more refined result for strong polynomial tractability. 
With a suitable choice of weights we can get rid of the dependence on $s$ in Theorem~\ref{thmcbc}.
\begin{theorem}\label{thmweights}
Let $\mu>1$, let 
$\bsgamma = (\gamma_j)_{j \in \mathbb{N}}$ be a sequence of positive real numbers and assume that 
$\kappa \ge \sup_{j \in \mathbb{N}} \max\{\exp(\mathrm{e}^2),\exp(\mathrm{e}^2 \gamma_j^{1/\mu})\}$ and that for some 
$1/\mu < \lambda \le 1$ we have $$\Gamma:=\sum_{j=1}^\infty \gamma_j \max\{1,\log \gamma_j^{-1}\}^{\mu(1-\lambda)} < \infty.$$
Then for the generating vector $\bsg^\ast$ constructed by Algorithm~\ref{algcbc} 
and the worst-case error in the space $H_{K_{\log, \mu, \bsgamma}}$ 
we have $$e^2(\bsg^\ast) \ll_{\mu,\lambda,\kappa,\Gamma} \frac{1}{N (\log N)^{\mu (1- \lambda)}},$$ where
the implied constant is independent of the dimension $s$.
\end{theorem}

\begin{proof}
If $\Gamma:=\sum_{j=1}^\infty \gamma_j \max\{1,\log \gamma_j^{-1}\}^{\mu(1-\lambda)} < \infty$, then using an argument as for the derivation of~\eqref{ineq_prod_sum} yields
$$
T_s (\mu,\lambda,\kappa,\bsgamma)
\le 3 \exp\left(2C_{\mu,\lambda,\kappa} \Gamma\right)=:B_{\mu,\lambda,\kappa,\bsgamma}\ \ \ \mbox{ for all $s \in \NN$.}
$$
From~\eqref{e_bound}, we then have 
\begin{equation}\label{e_bound1}
e^2(\bsg^\ast)\le \varphi^{-1} \left(B_{\mu,\lambda,\kappa,\bsgamma}/N\right).
\end{equation}
Now choose $N\ge B_{\mu,\lambda,\kappa,\bsgamma}$ (note that the choice of $N$ does not depend on $s$), and again
apply Lemma \ref{lemphi}\ref{lemphi_inv} to~\eqref{e_bound1}. This yields
$$e^2(\bsg^\ast)\le\frac{B_{\mu,\lambda,\kappa,\bsgamma}}{N
(\log (N/B_{\mu,\lambda,\kappa,\bsgamma}))^{\mu(1-\lambda)}}.$$
The result follows.
\end{proof}

If $\sum_{j=1}^\infty \gamma_j^{\tau} < \infty$ for $\tau = 1$ but not for any $\tau < 1$, then, for instance, \cite[Theorem~4]{K03} yields a convergence of order $\mathcal{O}(N^{-1/2})$. 
On the other hand, Theorem~\ref{thmweights} can yield a slightly better convergence rate if 
$\sum_{j=1}^\infty \gamma_j \max\{1, \log \gamma_j^{-1}\}^{\mu (1-\lambda)} < \infty$ 
and therefore yields a more precise estimate for such cases. For instance, in \cite{KSS12} additional conditions are needed when $\tau=1$, but it is not clear whether this can be partly avoided using Theorem~\ref{thmweights}.

\section{A possible extension}\label{sec_ext}

The results in this paper could also be extended in the following way. 
Let $\log_i (x)$ denote the $i$ times iterated logarithm, that is, $\log_0 (x) = x$, 
$\log_1 (x) = \log x$, $\log_2 (x) = \log \log x$ and so on. Let $\kappa$ be a fixed, sufficiently large, real number. Define
\begin{equation*}
r_{\mu, \gamma, \ell}(k) = \left\{ \begin{array}{ll} 
1 & \mbox{if } k = 0,\\
\gamma  (\log_\ell (\kappa |k|))^{-\mu} \prod_{i=0}^{\ell-1} 
(\log_i (\kappa |k|))^{-1} & \mbox{otherwise,}
\end{array} \right.
\end{equation*}
and $r_{\mu, \bsgamma, \ell}(\boldsymbol{k}) = \prod_{j=1}^s r_{\mu, \gamma_j, \ell}(k_j)$. Note that for any $\ell \in \mathbb{N}$ and $\mu > 1$ we have
\begin{equation*}
\sum_{k=1}^\infty r_{\mu, \ell}(k) < \infty.
\end{equation*}

We can again define a reproducing kernel
\begin{equation*}
K_{\log_\ell, \mu, \bsgamma}(\boldsymbol{x}, \boldsymbol{y}) = \sum_{\boldsymbol{k} \in \mathbb{Z}^s} r_{\mu, \bsgamma, \ell}(\boldsymbol{k}) \mathrm{e}^{2\pi \mathrm{i} \boldsymbol{k} \cdot (\boldsymbol{x} - \boldsymbol{y})}.
\end{equation*}
Using a modification of the function $\varphi$, similar results for the space  $H_{K_{\log_\ell, \mu, \bsgamma}}$ as for the space  $H_{K_{\log, \mu, \bsgamma}}$ could be obtained. \\

\noindent {\bf Addresses:} \\

Josef Dick, School of Mathematics and Statistics, The University of New South Wales, Sydney, 2052 NSW, Australia. e-mail: josef.dick(AT)unsw.edu.au \\

Peter Kritzer, Gunther Leobacher, Friedrich Pillichshammer, Department of Financial Mathematics, Johannes Kepler University Linz, Altenbergerstr. 69,
4040 Linz, Austria. e-mail: {peter.kritzer, gunther.leobacher, friedrich.pillichshammer}(AT)jku.at


\begin{thebibliography}{99}

\bibitem{abst} M. Abramowitz and I.A. Stegun, {\it Handbook of Mathematical Functions.} Dover, New York, 1964.

\bibitem{Aron} N. Aronszajn, Theory of reproducing kernels. Trans. Amer. Math. Soc. 68, (1950), 337--404.

\bibitem{CKN06} R. Cools, F.Y. Kuo and D. Nuyens, Constructing embedded lattice rules for multivariable integration. SIAM J. Sci. Comput. 28, (2006), 2162--2188.

\bibitem{D04} J. Dick, On the convergence rate of the component-by-component construction of good lattice rules. J. Complexity 20, (2004), 493--522.

\bibitem{DKLP15}  J. Dick, P. Kritzer, G. Leobacher and F. Pillichshammer, A reduced fast component-by-component construction of lattice points for integration in weighted spaces with fast decreasing weights. J. Comput. Appl. Math., 276 (2015),  1--15.

\bibitem{DKS13} J. Dick, F.Y. Kuo, and I.H. Sloan, High-dimensional integration: the quasi-Monte Carlo way. Acta Numer. 22, (2013), 133--288.

\bibitem{DNP} J. Dick, D. Nuyens and F. Pillichshammer, Lattice rules for nonperiodic smooth integrands. Numer. Math. 126, (2014), 259--291.

\bibitem{DP10} J. Dick and F. Pillichshammer: {\it Digital Nets and Sequences. Discrepancy Theory and Quasi-Monte Carlo Integration.} Cambridge University Press, Cambridge, 2010.

\bibitem{Hick} F.J. Hickernell, A generalized discrepancy and quadrature error bound. Math. Comp. 67, (1998),  299--322.

\bibitem{Hick02} F.J. Hickernell, Obtaining $O(N^{-2+\varepsilon})$ convergence for lattice quadrature rules. 
In: K.T. Fang, F.J. Hickernell, H. Niederreiter (eds.), Monte Carlo and quasi-Monte Carlo methods, 2000 (Hong Kong),  274--289, Springer, Berlin, 2002.

\bibitem{Kor59} N.M. Korobov, The approximate computation of multiple integrals (in Russian), Dokl. Akad. Nauk SSSR 124, (1959), 1207--1210.


\bibitem{K03} F.Y. Kuo, Component-by-component constructions achieve the optimal rate of convergence for multivariate integration in weighted Korobov and Sobolev spaces. Numerical integration and its complexity (Oberwolfach, 2001). J. Complexity 19, (2003), 301--320.

\bibitem{KSS12} F.Y. Kuo, Ch. Schwab and I.H. Sloan, Quasi-Monte Carlo finite element methods for a class of elliptic partial differential equations with random coefficients. SIAM J. Numer. Anal. 50, (2012), 3351--3374.

\bibitem{LP14} G. Leobacher and F. Pillichshammer, {\it Introduction to quasi-Monte Carlo integration and applications.} Birkh\"auser/Springer, 2014.

\bibitem{N92} H. Niederreiter, {\it Random number generation and quasi-Monte Carlo methods.} CBMS-NSF Regional Conference Series in Applied Mathematics, 63. Society for Industrial and Applied Mathematics (SIAM), Philadelphia, PA, 1992.

\bibitem{NW08} E. Novak and H. Wo\'{z}niakowski, {\it Tractability of Multivariate Problems, 
Volume 1: Linear Information.} European Mathematical Society, Z\"urich, 2008.

\bibitem{NW10} E. Novak and H. Wo\'{z}niakowski, {\it Tractability of Multivariate Problems, 
Volume 2: Standard Information for Functionals.} European Mathematical Society, Z\"urich, 2010.

\bibitem{NC06} D. Nuyens and R. Cools, Fast algorithms for component-by-component construction of rank-1 lattice rules in shift-invariant reproducing kernel Hilbert spaces. Math. Comp. 75, (2006), 903--920.

\bibitem{NC06b} D. Nuyens and R. Cools, Fast component-by-component construction, a reprise for different kernels. In: A. Keller, S. Heinrich, H. Niederreiter (eds.), Monte Carlo and quasi-Monte Carlo methods 2004, 373--387, Springer, Berlin, 2006.

\bibitem{SKJ02} I.H. Sloan, F.Y. Kuo and S. Joe, Constructing randomly shifted lattice rules in weighted Sobolev spaces. SIAM J. Numer. Anal. 40, (2002), 1650--1665.

\bibitem{SR02} I.H. Sloan and A.V. Reztsov, Component-by-component construction of good lattice rules. Math. Comp. 71, (2002), 263--273.

\bibitem{SW01} I.H. Sloan and H. Wo\'zniakowski, Tractability of multivariate integration for weighted Korobov classes. Complexity of multivariate problems (Kowloon, 1999). J. Complexity 17, (2001), 697--721.

\end{thebibliography}
\end{document}